\documentclass[11pt,a4paper]{amsart}

\usepackage{amsmath,amsfonts,amssymb,amsthm}
\usepackage{mathtools}
\usepackage{epsfig}
\usepackage{graphicx}
\usepackage{url}
\usepackage[usenames,dvipsnames]{xcolor}
\usepackage[colorlinks=true,linkcolor=Blue,citecolor=Green,pdfpagelabels]{
hyperref}
\usepackage{cleveref}
\usepackage{nicefrac}
\usepackage{aliascnt} 
\usepackage[alwaysadjust]{paralist}
\usepackage[T1]{fontenc}
\usepackage[utf8]{inputenc}
\usepackage{tikz,xifthen}
\usepackage{algorithm,algorithmic,eqparbox}

\floatname{algorithm}{} 

\newtheorem{theorem}{Theorem}[section]
\newaliascnt{lemma}{theorem}
\newaliascnt{corollary}{theorem}
\newaliascnt{definition}{theorem}
\newaliascnt{remark}{theorem}
\newaliascnt{proposition}{theorem}
\newaliascnt{conjecture}{theorem}
\newaliascnt{example}{theorem}
\newaliascnt{problem}{theorem}

\newtheorem*{theorem*}{Theorem}

\newtheorem{lemma}[lemma]{Lemma}
\aliascntresetthe{lemma}

\newtheorem*{lemma*}{Lemma}

\aliascntresetthe{corollary}

\newtheorem*{corollary*}{Corollary}

\aliascntresetthe{definition}

\newtheorem*{definition*}{Definition}

\newtheorem{remark}[remark]{Remark}
\aliascntresetthe{remark}

\newtheorem*{remark*}{Remark}

\newtheorem{proposition}[proposition]{Proposition}
\aliascntresetthe{proposition}

\newtheorem*{proposition*}{Proposition}

\newtheorem{conjecture}[conjecture]{Conjecture}
\aliascntresetthe{conjecture}

\newtheorem*{conjecture*}{Conjecture}

\aliascntresetthe{example}

\newtheorem*{example*}{Example}

\aliascntresetthe{problem}

\newtheorem*{problem*}{Problem}

\DeclareMathOperator{\inter}{int}
\DeclareMathOperator{\relint}{relint}

\DeclareMathOperator{\conv}{conv}

\def\R{\mathbb{R}}

\def\Z{\mathbb{Z}}

\def\N{\mathbb{N}}
\def\Q{\mathbb{Q}}
\def\cA{\mathcal{A}}
\def\cB{\mathcal{B}}
\def\cF{\mathcal{F}}
\def\O{\mathcal{O}}
\DeclareMathOperator{\vol}{vol}

\DeclareMathOperator{\flt}{Flt}
\DeclareMathOperator{\ML}{ML}
\DeclareMathOperator{\NP}{NP}
\DeclareMathOperator{\coNP}{coNP}

\newcommand{\zero}{\mathbf{0}}
\newcommand{\covradname}[1]{\ensuremath{\mathtt{CoveringRadius}(#1)}}

\usepackage[colorinlistoftodos,textwidth=2.5cm,textsize=small]{todonotes}

\newcommand{\iim}{i\in[m]}
\newcommand{\Np}{L_P}
\newcommand{\Npbar}{\bar L_P}
\newcommand{\Ren}{\R^n}
\newcommand{\Zen}{\Z^n}
\newcommand{\cube}{[0,1]^n}
\newcommand{\normp}{\|P\|_{\infty}}

\numberwithin{equation}{section}

\newcommand{\La}{\ensuremath{\Lambda}}
\newcommand{\Ga}{\ensuremath{\Gamma}}

\newcommand{\la}{\ensuremath{\lambda}}

\newcommand{\set}[1]{{\left\{{#1}\right\}}}
\newcommand{\eps}{\ensuremath{\varepsilon}}

\begin{document}

\title[Computing the covering radius of a polytope]{Computing the covering radius of a polytope with an application to lonely runners}


\author[J. Cslovjecsek]{Jana Cslovjecsek}
\address{Institute for Mathematics\\
         \'{E}cole Polytechnique F\'{e}d\'{e}rale de Lausanne\\
         Lausanne\\
         Switzerland}
\email{jana.cslovjecsek@epfl.ch}

\author[R. Malikiosis]{Romanos Diogenes Malikiosis}
\address{Department of Mathematics\\
		 Aristotle University of Thessaloniki\\
         Thessaloniki\\
         Greece}
\email{rwmanos@gmail.com}

\author[M. Nasz\'{o}di]{M\'{a}rton Nasz\'{o}di}
\address{Alfr\'ed R\'enyi Institute of Mathematics\\
MTA-ELTE Lend\"ulet Combinatorial Geometry Research Group\\
Dept. of Geometry, Lor\'and E\"otv\"os University\\
         Budapest\\
         Hungary}
\email{marton.naszodi@math.elte.hu}

\author[M. Schymura]{Matthias Schymura}
\address{Institut f\"ur Mathematik\\
  BTU Cottbus-Senftenberg\\
  Platz der Deutschen Einheit 1\\
  D-03046 Cottbus\\
  Germany}
\email{schymura@b-tu.de}

\thanks{JC and MS were supported by the Swiss National Science Foundation (SNSF) within the project \emph{Lattice Algorithms and Integer Programming (Nr.~185030)}.
MN was supported by the National Research, Development and
Innovation Fund (NRDI) grants K119670 and KKP-133864 as well as the Bolyai Scholarship of the Hungarian Academy of Sciences and the TKP2020-NKA-06 program provided by the NRDI}



\begin{abstract}
We study the computational problem of determining the covering radius of a rational polytope.
This parameter is defined as the minimal dilation factor that is needed for the lattice translates of the correspondingly dilated polytope to cover the whole space.
As our main result, we describe a new algorithm for this problem, which is simpler, more efficient and easier to implement than the only prior algorithm of Kannan (1992).

Motivated by a variant of the famous Lonely Runner Conjecture, we use its geometric interpretation in terms of covering radii of zonotopes, and apply our algorithm to prove the first open case of three runners with individual starting points.
\end{abstract}

\maketitle

\section{Introduction}

Let $K$ be a \emph{convex body}, that is, a compact convex subset of the $n$-dimensional Euclidean vector space~$\R^n$, and let $\Lambda \subseteq \R^n$ be a full-dimensional \emph{lattice}, that is, a discrete subgroup therein.
The \emph{covering radius} $\mu(K,\Lambda)$ of~$K$ with respect to~$\Lambda$ is the smallest non-negative real number~$\mu$ such that the lattice arrangement $\mu K+\Lambda = \bigcup_{z \in \Lambda} (\mu K + z)$ of~$\mu K$ is a covering of~$\Ren$, that is, $\mu K+\Lambda=\Ren$.
Equivalently, $\mu(K,\Lambda)$ is the maximal~$\mu >0$ such that $\mu K$ can be translated to a \emph{lattice-free} position, meaning a position in which the body contains no point of the lattice~$\Lambda$ in its interior (see~\cite[Sect.~13.1]{gruberlekkerkerker1987geometry} for details).
If $\Lambda=\Z^n$ is the standard lattice then we write $\mu(K)=\mu(K,\Z^n)$ for brevity.

The covering radius is a classical and much-studied parameter in the Geometry of Numbers, in particular in the realm of transference results, the reduction of quadratic forms, and Diophantine Approximations (cf.~\cite{gruberlekkerkerker1987geometry} for background).
The study of this geometric concept was revived with Lenstra's landmark paper~\cite{lenstra1983integer} on solving Linear Integer Programming in fixed dimension in polynomial time.
Lenstra's ideas were based on the famous \emph{flatness theorem}, which quantifies the intuition that lattice-free convex bodies are flat in some direction.
Stronger bounds in the flatness theorem with new applications in Number Theory were developed by Kannan \& Lov\'{a}sz~\cite{kannanlovasz1988covering} soon after.
More recent applications of the covering radius include (a) the classification of lattice polytopes, in particular lattice simplices, in small dimensions (cf.~Iglesias-Vali\~{n}o \& Santos~\cite{iglesiassantos2019classification} and the references therein), (b) distances between optimal solutions of mixed-integer programs and their linear relaxations (Paat, Weismantel \& Weltge~\cite{paatweismantelweltge2020distances}), (c) unique-lifting properties of maximal lattice-free polyhedra (Averkov \& Basu~\cite{averkovbasu2015lifting}), and (d) another viewpoint on the famous Lonely Runner Problem (cf.~\cite{henzemalikiosis} and Section~\ref{sect:zonotopal-lrc} in the paper at hand).

Despite these versatile applications, the question of how to actually compute the covering radius of a given convex body has not received much attention in the literature.
This is probably due to the immense computational hardness of the problem:
Kannan~\cite{kannan1992latticetranslates} reduced the classical Frobenius coin exchange problem, which is known to be $\NP$-hard, to computing the covering radius of certain simplices.
In the same paper, Kannan described an algorithm for the covering radius of a rational polytope, which however has a time complexity that involves a double exponentiation of the input size (see \Cref{thm:alg-kannan} for a precise statement).
Moreover, Haviv \& Regev~\cite{havivregev2012hardness} showed that it is even $\Pi_2$-hard to approximate the covering radius of linear images of the unit ball of the $p$-norm within a constant factor and for large enough values of $p \leq \infty$.
Similarly to $\coNP$ being the complement of~$\NP$, the class~$\Pi_2$ is the complement of the class of decision problems that can be solved in polynomial time by a non-deterministic algorithm that, additionally, has access to an $\NP$-oracle.
It holds that $\NP \subseteq \Pi_2$ and $\coNP \subseteq \Pi_2$, and no $\Pi_2$-hard problem belongs to~$\NP$, unless the polynomial time hierarchy collapses.
Micciancio~\cite{micciancio2004almostperfect} (cf.~Micciancio \& Goldwasser~\cite{miccianciogoldwasser2002complexity}) used hardness results of (variants of) the covering radius problem in the context of designing secure cryptosystems in lattice-based cryptography.

The main objective in this paper is to consider Kannan's work and devise a simpler and at the same time more efficient algorithm for computing the covering radius of a given rational polytope.
As our main result we obtain the following (see \Cref{thm:algfast} for the precise statement):

\begin{theorem*}
\label{thm:main-intro}
Let $P \subseteq \R^n$ be a rational polytope with $m$ facets and input size bounded by~$\Delta$.
Then, there is an algorithm that computes the covering radius~$\mu(P)$ of $P$ in time
\begin{equation*}
 \O\left((\Delta \cdot n)^{2n^2(n+2)} \cdot m^{n+2}\right).
\end{equation*}
\end{theorem*}

The devised algorithm is based on a description of the covering radius in terms of certain \emph{last-covered} points, which are points that are not contained in the interior of any lattice translate of~$\mu(P)P$ (see \Cref{lem:lastcovered}).
Together with the periodicity of the lattice and the boundedness of the polytope, this reduces the task to solving finitely many systems of linear equations and inequalities.

Our original motivation to study the computability of the covering radius was drawn from an application to the famous \emph{Lonely Runner Conjecture}.
Originally stated by Wills~\cite{willslonelyrunner} in the 1960's as a problem in Diophantine Approximation, it is probably best known via Goddyn's interpretation:
Consider~$d$ runners that run around a circular track of length~$1$ with pairwise distinct constant velocities.
The claim is that there is a time at which every runner has a distance of at least $1/(d+1)$ from the common starting point.

This is a notoriously difficult problem that received renewed attention in the literature after Tao~\cite{tao2018someremarks} made the first significant progress after many years and improved the known bounds on the guaranteed distance that the runners can achieve simultaneously.
In whole generality the conjecture is proven for up to $d\leq6$ runners (cf.~\cite{barajasserra2009} and the references therein).
In recent years people started wondering whether the problem is posed in the most natural way.
For example, in~\cite{beckhostenschymura2018lrpolyhedra} it is proposed that the restriction that the runners all start at the same place might very well be superficial.
Relaxing this condition and assuming each runner to start at an individual position leads to what we call the \emph{Shifted} Lonely Runner Conjecture (see \Cref{conj:sLRC} for details).
Another stronger formulation that could be well-suited for inductive approaches is due to Kravitz~\cite{kravitz2019barely} and called the \emph{Loneliness Spectrum Conjecture}.

We settle the first open case of the Shifted Lonely Runner Conjecture, that is, the case of three runners.
The precise statement of the following theorem, including a characterization of the extremal triples of velocities, is given as \Cref{thm:sLRC-dim3} in Section~\ref{sect:geometric-reduction-three-runners}.

\begin{theorem*}
\label{rhm:sLRC-dim3-intro}
Consider three runners with pairwise distinct constant velocities, who start running on a circular track of length $1$, with not necessarily identical starting positions.
A stationary spectator watches the runners from a fixed position along the track.
Then, there exists a time at which all the runners have distance at least $1/4$ from the spectator.
\end{theorem*}

The interpretation of the Lonely Runner Problem via covering radii of certain zonotopes was established in~\cite{henzemalikiosis}.
Based on this geometric interpretation, the main idea of the proof of the above result is to reduce the problem to only a small list of concrete triples of velocities and then either proceed by hand or use the developed algorithm to compute the exact covering radius of the particular instances.

\subsection*{Organization of the paper}

For the reader's convenience we first give a short review of Kannan's original approach to compute the covering radius of a rational polytope.
Afterwards in Section~\ref{sect:novel-algorithm}, we describe the details of our new algorithm and derive a bound on its time complexity that beats Kannan's complexity estimation (in Theorem~\ref{thm:alg-kannan}) for every reasonable input.
Finally in Section~\ref{sect:lonely-runners}, we explain the relationship between the Lonely Runner Conjecture and the computation of the covering radius of certain zonotopes, and apply our algorithm to prove the Shifted Lonely Runner Conjecture for three runners.

\section{A short review of Kannan's approach}

Throughout, we let $A \in \R^{m \times n}$ and $b \in \R^m$ be the defining data of a full-dimensional polytope $P = \{x \in \R^n : A x \leq b\}$.
The rows of $A$ are denoted by $a_1,\ldots,a_m$ and the entries of $b$ by $b_1,\ldots,b_m$.

Kannan's algorithm for the computation of $\mu(P)$ is based on the fact that, if both $A$ and $b$ are rational, then $\mu(P)$ is a rational number whose numerator and denominator are polynomially bounded by the size of the input (see~\cite[Prop.~(5.1)]{kannan1992latticetranslates}).
This allows to use a binary search if one is able to decide whether the arrangement $\mu P + \Z^n$ covers the whole space~$\R^n$, for any fixed $\mu \in \Q_{\geq0}$.
Kannan's approach to this decision problem is to derive a polyhedral description of the arrangement $\mu P+\Z^n$ in terms of finite data.
In his main structural result, \cite[Thm.~(4.1)]{kannan1992latticetranslates}, he shows that there is a partition of~$\R^n$ into polyhedral $\Z^n$-periodic sets $S_1,\ldots,S_r$, meaning that $S_i + \Z^n = S_i$, such that for each $S_i$ there is a family $\cB_i$ of bases of~$\Z^n$, and for each basis $B \in \cB_i$ there is a finite subset $Z_B \subseteq \Z^n$ and an affine transformation $T_B:\R^m\to\R^n$ such that, for each $1 \leq i \leq r$,
\begin{align}
(\mu P+\Z^n) \cap S_i &= \left(\Bigg(\bigcup_{B \in \cB_i}(\mu P+Z_B)\cap(T_B\cdot b + F_B)\Bigg) + \Z^n\right) \cap S_i,\label{eq:kannan-description}
\end{align}
where $F_B = B [0,1)^n$ denotes the fundamental cell of~$\Z^n$ associated with the basis~$B$.
Remember that $b$ is the right hand side of the inequality system that defines~$P$.
Moreover, all the involved objects $S_i,\cB_i,Z_B,T_B,F_B$ can be explicitly computed and are bounded in size only by a function in the size of the input.
Note that in the very special situation that $\mu P$ is contained in a fundamental cell $F_B$ of~$\Z^n$, this description reduces to $\mu P+\Z^n = (\mu P \cap F_B) + \Z^n$, which has been the starting point of Kannan's investigation.

With the description~\eqref{eq:kannan-description} at hand, he now formulates a family of mixed-integer linear programs of bounded size that model the question whether, for a given $1 \leq i \leq r$, there exists a point $x \in S_i$ such that $x \notin \mu P+\Z^n$.
In fact, this decision boils down to checking whether, for every basis $B \in \cB_i$ and every lattice point $z \in Z_B$, the unique lattice translate $x_B$ of a given point $x \in S_i$, that is contained in the parallelepiped $F_B + T_B \cdot b$, is contained in the translate $\mu P + z$.
Employing algorithms for checking feasibility of mixed-integer linear programs then facilitates the binary search for computing~$\mu(P)$ as described above.

\begin{theorem}[{Kannan~\cite{kannan1992latticetranslates}}]
\label{thm:alg-kannan}
Let $P = \{x \in \R^n : a_i^\intercal x \leq b_i, \iim\}$ be a rational polytope, with $a_i \in \Zen$ and $b_i \in \Z_{>0}$, for all $\iim$.
Then, there is an algorithm that computes the covering radius of $P$ in time
\begin{equation*}
 (n m \log \normp)^{n^{\O(n)}},
\end{equation*}
where $\normp$ is the maximal absolute value of an entry of the data~$a_i$ and~$b_i$.
\end{theorem}


Although, in the proof of~\cite[Prop.~(5.1)]{kannan1992latticetranslates}, Kannan hinted at our main structural result in Lemma~\ref{lem:lastcovered} below, he did not follow that path for the computation of~$\mu(P)$.
One may speculate that the reason was that already while writing~\cite{kannan1992latticetranslates} he had an extension of his methods in mind, that he used in~\cite{kannan1990testsets} to design a decision procedure for sentences of the form
\[
\forall\, y \in \Z^p\ \exists\, x \in \Z^n : Ax + By \leq b.
\]
Also, he was interested in allowing the right hand side $b\in\R^m$ in the definition of~$P$ to vary in specified polyhedral regions.
The decomposition technique in~\eqref{eq:kannan-description} is suited for this purpose, and we refer to Kannan's own explanations in~\cite[p.~163]{kannan1992latticetranslates} for more information.

\section{Enter Geometry: A simpler and faster algorithm}
\label{sect:novel-algorithm}

Here, we present an algorithm to compute the covering radius of a rational 
polytope, that is different in spirit from the one that 
Kannan~\cite{kannan1992latticetranslates} gave, is easy to implement, and has a 
superior running time for every reasonable input; a quantitative comparison is presented at the end of this section.
Our algorithm is based on a crucial geometric observation that we discuss first.

Given a convex body $K \subseteq \R^n$, we call a point $p \in \Ren$ \emph{last-covered} by~$K$, if $p \notin \inter(\bar\mu K) + \Z^n$, where $\bar\mu = \mu(K)$.
The intuition behind this notion is that if we consider the lattice arrangements $\mu K + \Z^n$, for increasing values of $\mu \leq \bar \mu$, then~$p$ does not belong to any lattice translate of $\mu K$ unless $\mu = \bar \mu$.
This concept is very natural for the investigation of the covering radius and already appeared before, for instance in~\cite{codenottisantosschymura2019the}.
Figure~\ref{fig:last-covered-points} illustrates the concept of last-covered points on a covering by triangles and another one by squares.

\begin{figure}[th]
\hfill\includegraphics[scale=.7,page=3]{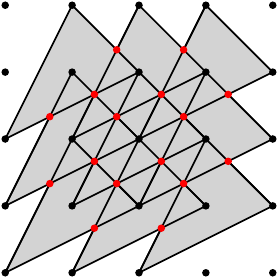}
\hfill\includegraphics[scale=.7,page=6]{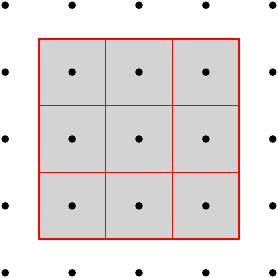}
\hfill\,
\caption{A portion of a covering of the plane by translates of a triangle (left) and a square (right). The last-covered points are illustrated in red.}
\label{fig:last-covered-points}
\end{figure}

\begin{lemma}
\label{lem:lastcovered}
Let $P = \left\{x \in \Ren: a_i^\intercal x \leq b_i, \iim\right\}$ be a facet description of the polytope $P$ with $b_i>0$, for all $i \in [m]$.\footnote{This assumption means that the origin is contained in the interior of~$P$, which is no restriction as $\mu(P+t) = \mu(P)$, for every $t \in \R^n$.}
Then, there are facet normals $a_{i_1}, \ldots, a_{i_{n+1}}$ of~$P$ 
and not necessarily distinct lattice points $z_1,\ldots,z_{n+1} \in \Z^n$ 
such that the system of linear equations
\begin{align}
\mu &= {a_{i_1}^\intercal (x - z_1)}/b_{i_1} = \ldots = 
{a_{i_{n+1}}^\intercal (x - z_{n+1})}/b_{i_{n+1}}\label{eq:lastcovered}
\end{align}
in the variables $\mu$ and $x$ has a unique solution $(\bar \mu, \bar p)$, where $\bar \mu = \mu(P)$ and~$\bar p$ is a last-covered point with respect to~$P$.
\end{lemma}

\begin{proof}
After a suitable scaling we may assume that $b_1=\ldots=b_m=1$, and we write $\bar \mu = \mu(P)$ for brevity.
For a last-covered point $p \in \R^n \setminus \left(\inter(\bar \mu P) + \Z^n\right)$, we let
\[
\cF_p := \left\{F : \exists\,z \in \Z^n\text{ such that }F\text{ is a facet of }\bar \mu P + z\text{ that contains }p\right\}
\]
be the set of facets of lattice translates of $\bar \mu P$ that contain~$p$, and we write $\varphi(p) = \# \cF_p$ for its cardinality.
This number is clearly finite.
In fact, $\varphi(p)\leq m\cdot\max\{\#((q-\bar\mu P)\cap\Z^n) : q\in\Ren\}$, where the maximum is finite since $P$ is bounded and the lattice~$\Z^n$ is periodic.
Thus, there is a last-covered point~$\bar p$ such that $\varphi(\bar p) \geq \varphi(p)$, for every $p \in \R^n \setminus \left(\inter(\bar \mu P) + \Z^n\right)$.

Now, let $k=\varphi(\bar p)$ and $\cF_{\bar p} = \{F_1,\ldots,F_k\}$ be the family of facets of lattice translates of $\bar \mu P$ that contain $\bar p$, and for each $ 1 \leq j \leq k$, let $a_{i_j}^\intercal(x - z_j) = \bar \mu$ be the linear equation of the supporting hyperplane of~$F_j$, where the $z_1,\ldots,z_k \in \Z^n$ are suitable not necessarily distinct lattice points.
We subsume these equations as the system
\begin{align}
\mu &= {a_{i_1}^\intercal (x - z_1)} = \ldots = 
{a_{i_{k}}^\intercal (x - z_{k})}\label{eq:lastcoveredA}
\end{align}
in the variables $(\mu,x) \in \R^{n+1}$, and consider the $z_j$ and the $a_{i_j}$ to be fixed in the sequel.
By construction $(\bar \mu, \bar p)$ is a solution of~\eqref{eq:lastcoveredA}.

Let $\cA := \{a_{i_1},\ldots,a_{i_k}\}$ be the set of facet normals of~$P$ involved in the linear system~\eqref{eq:lastcoveredA}.

\emph{Claim:} The affine hull of $\cA$ equals $\Ren$.

\noindent Indeed, suppose that $\cA$ is contained in the hyperplane $\{y \in \Ren: g^\intercal y = \gamma\}$, for some 
$g\in\Ren \setminus\{\zero\}$ and $\gamma\geq 0$.
We distinguish cases according to the sign of 
$\gamma$.
If $\gamma>0$, then for every $\varepsilon>0$ and every $1 \leq j \leq k$, we have
\[
a_{i_j}^\intercal(\bar p+\varepsilon g-z_j) = \bar \mu + \varepsilon a_{i_j}^\intercal g = \bar \mu + \varepsilon \gamma > \bar \mu,
\]
and therefore the point $\bar p+\varepsilon g$ is 
not contained in any of the lattice translates of $\bar \mu P$ that contain $\bar p$.
If~$\varepsilon$ is small enough, then $\bar p+\varepsilon g$ is moreover not contained in any other lattice translate of $\bar\mu P$ either, contradicting the definition of the covering radius $\bar \mu = \mu(P)$.

If $\gamma=0$, then consider the ray $\{\bar p+tg : t \geq 0\}$.  By the compactness of the facets~$F_j$, and the fact that the set of last-covered points is closed, there is a largest $\bar t\geq 0$ for which $\{\bar p+\bar tg : 0 \leq t \leq \bar t\} \subseteq \bigcap_{j=1}^k F_j$ and $\{\bar p+\bar tg : 0 \leq t \leq \bar t\}$ consists only of last-covered points.
Observe that $\bar t > 0$ because, by the choice of~$\bar p$ and the assumption $\cA \subseteq \{y \in \R^n : g^\intercal y = 0\}$, we have $\bar p \in \relint(\bigcap_{j=1}^k F_j)$.
As $\bar p + \bar t g$ is a last-covered point and $\bar t$ is maximal, there must be an additional facet of a lattice translate of $\bar \mu P$ that contains $\bar p + \bar t g$, contradicting the maximality of~$\varphi(\bar p)$.

Thus, in either case we arrive at a contradiction and the claim is proven.

Now, we can show that the solution $(\bar \mu,\bar p)$ of the system~\eqref{eq:lastcoveredA} is unique. Indeed, by the above claim, $\eqref{eq:lastcoveredA}$ is a system of linear equations in the $n+1$ variables $(\mu,p)$, whose coefficient matrix has rank~$n+1$. The uniqueness of the solution thus follows, and the proof of Lemma~\ref{lem:lastcovered} is complete.
\end{proof}

%

In order to turn \Cref{lem:lastcovered} into an algorithm for~$\mu(P)$, we need to control the relevant lattice points $z_1,\ldots,z_{n+1}$ that determine the system~\eqref{eq:lastcovered} by the data that defines the polytope~$P$.
To this end, assume that we have an upper bound $\mu_0$ on the covering radius of $P$, that is, $\mu(P)\leq \mu_0$.
Clearly, $\mu P+\Zen$ is a covering of $\Ren$ if and only if the cube $\cube$ (which is a fundamental cell of $\Zen$) is contained in $\mu P+\Zen$. 
We let
\begin{equation}\label{eq:npdefinition}
 \Np:=\Zen\cap (\cube-\mu_0 P),
\end{equation}
which is the set of those lattice points $z \in \Z^n$ such that the translate $z + \mu_0 P$ has a non-empty intersection with~$[0,1]^n$.
Then, for any $\mu>0$ and because we assume that $\zero \in \inter(P)$, we have that 
\begin{equation*}
 \mu P+\Zen=\Ren \quad\mbox{ if and only if }\quad
 \mu P+\Np \supseteq \cube, 
\end{equation*}
or, equivalently,
\begin{equation}\label{eq:npisenoughB}
 \mu > \mu(P) \quad\mbox{ if and only if }\quad
 \inter(\mu P)+\Np \supseteq \cube. 
\end{equation}
Clearly, there is a last-covered point $p$ in $\cube$, and for this $p$, the 
lattice points $z_1,\ldots,z_{n+1}$ in \Cref{lem:lastcovered} may be chosen to belong to~$\Np$. 

In summary, we have shown the following concretization of 
\Cref{lem:lastcovered}.

\begin{proposition}
\label{prop:proofofalgorithm}
Let $P = \{x \in \R^n: a_i^\intercal x \leq b_i, \iim\}$ be a facet 
description of the polytope $P$ with $b_i>0$, for all $i \in [m]$.
Assume that $\mu(P)\leq \mu_0$ for some $\mu_0>0$, and let $\Np$ be 
defined by \eqref{eq:npdefinition}.

Then, there are facet normals $a_{i_1}, \ldots, 
a_{i_{n+1}}$ of $P$, 
and not necessarily distinct lattice points $z_1,\ldots,z_{n+1} \in \Np$ 
such that the system of linear equations \eqref{eq:lastcovered}
in the variables $\mu$ and $x$ has a unique solution $(\bar \mu, \bar p)$, and in this solution $\bar \mu = \mu(P)$ and~$\bar p$ is a last-covered point with respect to~$P$.
\end{proposition}

As the set $\Np$ may be difficult to compute, we consider a superset 
$\Npbar \supseteq \Np$ using a \emph{bounding box} for $P$.
For this purpose, let $\beta(P)$ be defined as
\begin{equation*}
 \beta(P):=\max\{\|x\|_{\infty}: x\in P\}=\max\{\|v\|_{\infty}: v 
 \mbox{ is a vertex of } P\},
\end{equation*}
where $\|x\|_{\infty}$ denotes the maximum norm of the vector~$x \in \R^n$.
Assume that we are provided with upper bounds $\beta(P)\leq \beta_0$ and  $\mu(P)\leq \mu_0$.
Then, we may define
\begin{equation}\label{eq:npbardefinition}
 \Npbar:=\Zen\cap (\cube-\mu_0 [-\beta_0,\beta_0]^n)
\end{equation}
and by the definition of~$\Np$ we readily get $\Npbar \supseteq \Np$.

After these preparations, we are ready to describe our algorithm for the computation of the covering radius~$\mu(P)$ of the polytope $P$.
As input, it takes the facet description of $P$, upper bounds $\beta_0,\mu_0$ as above, and it defines $\Npbar$ according to this input.
Notice that the system \eqref{eq:lastcovered} has a unique solution if and only if
the vectors $(a_{i_1}/b_{i_1},1),\ldots,(a_{i_{n+1}}/b_{i_{n+1}},1) \in \R^{n+1}$ are 
linearly independent, or equivalently, if $a_{i_1}/b_{i_{1}},\ldots,a_{i_{n+1}}/b_{i_{n+1}} \in \R^n$ are affinely independent.

\begin{algorithm}[H]
\caption{$\covradname{P, \beta_0,\mu_0}$}
\begin{algorithmic}[1]
 \STATE $\mu_{\mbox{max}} := 0$
 \FOR{$z_1,\ldots,z_{n+1}$ not necessarily distinct points in $\Npbar$}\label{alg:for1}
  \FOR{$a_{i_1},\ldots,a_{i_{n+1}}$ facet normals of $P$}\label{alg:for2}
   \IF{$a_{i_1}/b_{i_{1}},\ldots,a_{i_{n+1}}/b_{i_{n+1}}$ are affinely independent}\label{alg:if1}
    \STATE solve the linear system \eqref{eq:lastcovered} to obtain $(\mu,p)$
    \IF[Is $(\mu,p)$ relevant?]{$p\in[0,1]^n$ \AND $p\notin\inter(\mu P)+\Npbar$}\label{alg:if2}
     \STATE $\mu_{\mbox{max}} := \max\{\mu, \mu_{\mbox{max}}\}$
    \ENDIF
   \ENDIF
  \ENDFOR
 \ENDFOR
 \RETURN $\mu_{\mbox{max}}$
\end{algorithmic}
\end{algorithm}

The correctness of this algorithm follows directly from the previous considerations.

\begin{proposition}
\label{prop:alggood}
Let $P = \{x \in \Ren: a_i^\intercal x \leq b_i, \iim\}$ be a
facet description of the polytope $P$ with $b_i>0$, for all $i \in [m]$.
Assume $\beta(P)\leq \beta_0$ and $\mu(P)\leq \mu_0$, for some $\beta_0,\mu_0>0$.

Then, \covradname{P, \beta_0,\mu_0} returns the covering radius of~$P$.
\end{proposition}

\begin{proof}
By \Cref{prop:proofofalgorithm}, the two for-loops in Lines~\ref{alg:for1} and~\ref{alg:for2} list all choices of 
$z_1,\ldots,z_{n+1}$ and $a_{i_1},\ldots,a_{i_{n+1}}$ that need to be 
considered.
By~\eqref{eq:npisenoughB} and the second 'if' condition in Line~\ref{alg:if2} of the algorithm, the returned value $\mu_{\mbox{max}}$ is at most $\mu(P)$.
On the other hand, again by \Cref{prop:proofofalgorithm}, the returned 
value $\mu_{\mbox{max}}$ is at least~$\mu(P)$.
\end{proof}

The algorithm \covradname{P, \beta_0,\mu_0} and \Cref{prop:alggood} are only applicable if we have access to upper bounds on $\beta(P)$ and $\mu(P)$.
Our next task is to provide such bounds based on the description 
of~$P$.
We note that until this point, we made no assumption on the vectors $a_i$ and the right hand sides~$b_i$.

Now and in the sequel, we assume that~$P$ is a \emph{rational} 
polytope given by an \emph{integer facet description}, that is, $P = \{x \in \R^n : a_i^\intercal x \leq b_i, \iim\}$, where $a_i\in\Zen$ and $b_i\in\Z_{>0}$, for all~$\iim$.
We use the notation $\normp$ to denote the largest absolute value of the 
parameters appearing in the description of~$P$ above, that is,
\begin{equation*}
 \normp:=\max\left(\{\|a_i\|_{\infty}:\iim\}\cup\{b_i:\iim\}\right).
\end{equation*}

\noindent\textbf{First, we bound $\beta(P)$.} Observe that if $v$ is a 
vertex of $P$, then 
\begin{equation}\label{eq:vertexdescription}
v \mbox{ is the solution of a linear system } A_I x = b_I, 
\end{equation}
for some index set $I\subseteq[m]$ with $|I|=n$, where $A_I$ denotes the 
invertible square matrix obtained by considering only the rows of 
$A=(a_i^\intercal)_{\iim}$ indexed by~$I$, and $b_I$ is the corresponding vector of right hand sides.
Since $A_I$ has integer entries, we apply Cramer's rule and obtain that
\begin{equation}\label{eq:bbound}
 \beta(P)\leq \normp^n n!.
\end{equation}
Note that $\beta(P) = \max\{\|x\|_{\infty} : x\in P\}$ can be computed exactly in polynomial time (see~\cite{mangasarianshiau1986avariable}), which can be applied for an efficient implementation of our algorithm.
However, for the theoretical analysis that we are carrying out in these lines, we rely on the bound~\eqref{eq:bbound}.

\medskip
\noindent\textbf{Second, we bound $\mu(P)$.}
To do that, we define $\alpha(P)$ as the smallest $\alpha>0$ such that $\alpha P$ is a \emph{lattice polytope}, that is, all vertices of $\alpha P$ are lattice points. 
We observe that
\begin{equation}\label{eq:alphabound}
 \alpha(P)\leq\normp^n n!,
\end{equation}
which follows from \eqref{eq:vertexdescription} and again Cramer's rule combined with the estimate $\det(A_I)\leq\normp^n n!$.
Regarding the covering radius, we claim that
\begin{equation}\label{eq:gammabound}
 \mu(P)\leq\alpha(P)n.
\end{equation}
Indeed, let $\Delta$ be a non-degenerate lattice simplex contained in 
$\alpha(P) P$, one of whose vertices is $\zero$, and let $\Lambda$ denote the sublattice of $\Zen$ generated by the vertices of $\Delta$. 
Furthermore, let $\Delta_0=\conv(\{\zero,e_1,\ldots,e_n\})$ denote the standard lattice simplex.
Then
\begin{equation*}
\frac{1}{\alpha(P)}\mu(P) = \mu(\alpha(P)P,\Zen) \leq \mu(\Delta,\Zen) \leq \mu(\Delta,\Lambda) = \mu(\Delta_0,\Zen) = n,
\end{equation*}
where the first inequality follows from the fact that $\Delta \subseteq 
\alpha(P) P$, the second from $\Lambda \subseteq \Z^n$, and the next 
equality simply expresses a change of basis.
The fact that $\mu(\Delta_0)=n$ is well-known, and also immediate to verify.
This completes the proof of~\eqref{eq:gammabound}.

Another possibility to obtain a useful upper bound on~$\mu(P)$ and which is also a bit more tailored to the actual shape of the input polytope is based on the celebrated Flatness Theorem.
It states that there is a constant $\flt(n)$, the so-called \emph{flatness constant}, such that for every polytope $P \subseteq \R^n$ we have
\begin{equation}\label{eq:flatness-thm}
\mu(P) \leq \flt(n) \cdot w(P)^{-1},
\end{equation}
where
\[
w(P) = \min_{z \in \Z^n \setminus \{ \zero \}} \left(\max_{x \in P} x^\intercal z - \min_{x \in P} x^\intercal z\right)
\]
denotes the \emph{lattice-width} of~$P$.
The question to determine the asymptotic behavior of $\flt(n)$ as a function of the dimension~$n$ has attracted much interest in Integer Programming and the Geometry of Numbers in recent decades.
The best-known bound to date follows from a work of Rudelson~\cite{rudelson2000distances} and reads $\flt(n) \leq c \, n^{4/3} \log^a n$, for some unspecified absolute constants~$c$ and~$a$.
The worse but explicit bound $\flt(n) \leq n^{5/2}$ can be found in Barvinok's book~\cite[Ch.~VII, Thm.~(8.3)]{barvinok2002acourse}.

In order to get an explicit upper estimate in the flatness bound~\eqref{eq:flatness-thm}, we need a method to compute the lattice-width of the given polytope~$P$.
This can either be done via the interpretation of $w(P)$ as the first successive minimum of $(P-P)^\star$, that is, the polar of the difference body of~$P$ (see~\cite[Lem.~(2.3)]{kannanlovasz1988covering} for details), or via the algorithm described in~\cite{charrierfeschetbuzer2011computing}.

Since the lattice-width of a lattice polytope is at least~$1$, the definition of the parameter $\alpha(P)$ implies that $w(P)^{-1} \leq \alpha(P)$.
This shows that the bounds~\eqref{eq:gammabound} and~\eqref{eq:flatness-thm} are quite similar, and that the latter might be better for particular examples of~$P$.
For our further analysis we however always use the simpler bound~\eqref{eq:gammabound}.

We now run our algorithm $\covradname{P,\beta_0,\mu_0}$ with the values $\beta_0=\normp^n n!$ and $\mu_0=n\normp^n n!$, which are valid choices in view of the inequalities~\eqref{eq:bbound}, \eqref{eq:alphabound} and~\eqref{eq:gammabound}.
Slightly overloading notation, we abbreviate 
\[\text{\covradname{P}} = \text{\covradname{P, \normp^n n!,n\normp^n n!}},\]
and add that the algorithm starts with computing $\normp$ from the description of $P$, so it has a definition of~$\Npbar$.

Working towards an estimate on the time complexity of this algorithm, we first write down a bound on the cardinality of the set~$\Npbar$ respecting the previous choice of parameters:
\begin{equation}\label{eq:npbound}
 |\Npbar|\leq  
 \big(2+2\beta_0\mu_0\big)^n\leq
 \big(2+2n\normp^{2n} (n!)^2\big)^n\leq
 (\normp n)^{2n^2}.
\end{equation}
The two for-loops in $\covradname{P}$ have combined $|\Npbar|^{n+1} \cdot 
m^{n+1}$ iterations.
We may ignore the time needed for checking the 'if' condition in Line~\ref{alg:if1} and for solving the linear system~\eqref{eq:lastcovered}, since this can be done in polynomially many steps and is thus inferior in complexity to checking the second 'if' condition in Line~\ref{alg:if2}.
Indeed, given a candidate solution $(\mu,p)$, checking whether $p \notin \inter(\mu P) + \Npbar$ takes $m|\Npbar|$ steps.

In total, the algorithm takes $\O\left(|\Npbar|^{n+2} \cdot m^{n+2}\right)$ steps and with~\eqref{eq:npbound} we get the following upper bounds on the time complexities:
\begin{align*}
&\O\left((4 \beta_0 \mu_0)^{n(n+2)} \cdot m^{n+2}\right) &&\text{for} \quad \text{\covradname{P,\beta_0,\mu_0}}, \text{ and}\\
&\O\left((\normp n)^{2n^2(n+2)} \cdot m^{n+2}\right) &&\text{for} \quad \text{\covradname{P}}.
\end{align*}
We summarize the investigations of this section into our main result:

\begin{theorem}
\label{thm:algfast}
Let $P = \{x \in \R^n: a_i^\intercal x \leq b_i, \iim\}$ be a rational polytope, with $a_i \in \Zen$ and $b_i \in \Z_{>0}$, for all $\iim$.
Then, the algorithm \covradname{P} returns the covering radius of $P$ in time
\begin{equation*}
 \O\left((\normp n)^{2n^2(n+2)} \cdot m^{n+2}\right).
\end{equation*}
\end{theorem}
  
\noindent Let us compare our result to the running time of Kannan's algorithm:
According to \Cref{thm:alg-kannan} his algorithm has time complexity
\[
(n m \log \normp)^{n^{\O(n)}}.
\]
With respect to~$n$ and $m$ we achieved a significant improvement since the dependence in our algorithm involves only one exponentiation, compared to a double exponentiation in Kannan's approach.
The dependence on~$\normp$ is only better in case that $\normp$ is of order at most~$n^{n^{\O(n)}}$, that is however, for all reasonably large input sizes.

\section{An application to Lonely Runners}
\label{sect:lonely-runners}

In this section, we use a geometric interpretation of the famous Lonely Runner Conjecture to illustrate the utility of a (more efficient) algorithm to compute the covering radius of a rational polytope.
The problem that we are concerned with was raised as a question on simultaneous Diophantine approximations by J\"org M.~Wills in the 1960's~\cite{willslonelyrunner}.

\begin{conjecture}[Lonely Runner Conjecture]
\label{conj:lrc-original}
Given pairwise distinct numbers $v_0, v_1, \ldots, v_d \in \R$, for each $0 
\leq i \leq d$ there exists a real number~$t$ such that for all $0 \le j \le 
d$, $i \ne j$, the distance of $t ( v_i - v_j )$ to the nearest integer is at 
least $\frac{ 1 }{ d+1 }$.
\end{conjecture}

Independently of Wills, the problem also arose as a view-obstruction question in a work of Cusick~\cite{cusick1973viewobstruction}.
The name \emph{Lonely Runner Conjecture} goes back to the following descriptive interpretation due to Goddyn (1998):
Consider $d+1$ runners going at different constant velocities $v_0,v_1,\ldots,v_d$ around a circular track of length $1$ (having started at the same place and time).
Then, the conjecture says that each of them will at some point have distance at least $\frac{1}{d+1}$ to all the other runners.

A more convenient formulation of the problem is based on the
observation that the distance of any two runners at any given time depends only on their relative velocities.
So we may pick a fixed runner, say the one with velocity~$v_0$, reduce the velocity of every runner by $v_0$ and consider only the loneliness of the first runner that is now stagnant.

\begin{conjecture}[Lonely Runner Conjecture]
\label{LRC}
Given pairwise distinct positive $v_1, v_2, \ldots, v_d \in \R$, there exists a 
real number $t$ such that for all $1 \le j \le d$ the distance of $t v_j$ to 
the nearest integer is at least $\frac{ 1 }{ d+1 }$.
\end{conjecture}

For background information and related literature on the Lonely Runner Problem we refer the interested reader to~\cite{kravitz2019barely} or the survey article~\cite{schymurawills2018dereinsame} (in german).


Wills~\cite{willslonelyrunner} stated without proof that \Cref{LRC} can be reduced to the case of only integral velocities $v_1,\ldots,v_d \in \Z_{>0}$.
The first written proof was presented by Bohman, Holzman \& Kleitman~\cite[Lem.~8]{bohmanholzmankleitman2001six}, however, it is inductive on the dimension, so, in order to confirm the Lonely Runner Conjecture in a given dimension, it assumes that it is shown in lower dimensions.
An independent proof of this reduction can be found in~\cite[Lem.~5.3]{henzemalikiosis}.
For this reason we always assume in the sequel that every runner runs with an integral velocity.
Since the problem is moreover invariant under simultaneous scalings of the velocities, we may also assume that $\gcd(v_1,\ldots,v_d)=1$.

The reduction of the problem to integer velocities is crucial for its reformulation in the Geometry of Numbers.
An interpretation of \Cref{LRC} as a question about the existence of lattice points in certain convex regions has only recently been worked out (see~\cite{henzemalikiosis,beckhostenschymura2018lrpolyhedra}).
We discuss some of its details in Section~\ref{sect:zonotopal-lrc}.

For pairwise distinct integral velocities, it was conjectured 
in~\cite[Conj.~1]{beckhostenschymura2018lrpolyhedra} (and in a personal communication by J\"org M.~Wills himself) that Wills' conjecture is equivalent to a seemingly more general variant.
The point is that the runners with non-zero velocities may start running at arbitrarily chosen points on the track rather than all starting from the same position.

\begin{conjecture}[Lonely Runners with Individual Starting Points]
\label{conj:sLRC}
Given pairwise distinct non-zero 
velocities $v_1,\ldots,v_d \in \R$ 
and arbitrary starting points $s_1,\ldots,s_d \in \R$, there is a real number~$t$ such that for all $1 \leq j \leq d$ the distance of $s_j + t v_j$ to the nearest integer is at least $\frac{1}{d+1}$.
\end{conjecture}

Similarly to \Cref{LRC}, this can be reduced to pairwise distinct integral velocities $v_1,\ldots,v_d \in \Z_{>0}$; the argument is along the lines of the proof of~\cite[Lem.~5.3]{henzemalikiosis} and, for the sake of completeness, we provide full details below in Section~\ref{sect:reduction-integer-velocities}.

\Cref{conj:sLRC} holds true for two non-stagnant runners, that is, for $d=2$ as shown in~\cite{beckhostenschymura2018lrpolyhedra}.
For any other $d \geq 3$ the question is open.
In particular, in other cases where \Cref{LRC} holds, for instance, for any $3 \leq d \leq 6$, it is not clear yet whether \Cref{conj:sLRC} is indeed equivalent.

The assumption that the velocities $v_i$ shall be pairwise distinct is crucial in Conjecture~\ref{conj:sLRC}.
If we allow repeating velocities, then we are in the setting of a theorem of Schoenberg~\cite{schoenberg1976extremum} (see~\cite[Thm.~4]{beckhostenschymura2018lrpolyhedra} for an alternative proof).

\begin{theorem}[{Schoenberg~\cite{schoenberg1976extremum}}]
\label{thm:schoenberg}
Given integral velocities $v_1,\ldots,v_d \in \Z_{>0}$ and starting points $s_1,\ldots,s_d \in \R$, there is a real number~$t$ such that for all $1 \leq j \leq d$ the distance of $s_j + t v_j$ to the nearest integer is at least $\frac{1}{2d}$. 

Furthermore, this bound cannot be improved for $v_1=\ldots=v_d=1$ and starting points $s_i = \frac{i-1}{d}$, for $1 \leq i \leq d$.
\end{theorem}

Our aim is to solve the three runner problem\footnote{According to Conjecture~\ref{conj:lrc-original}, one could see this actually as the four-runner problem, as we assume that one of the runners is stagnant, that is, it has zero speed.} of \Cref{conj:sLRC}, that is, the case $d=3$.
Before we review the geometric approach hinted at above, we first need to reduce this problem to integer velocities.

\subsection{Reduction of Conjecture~\ref{conj:sLRC} to integer velocities}	
\label{sect:reduction-integer-velocities}

First, we re\-phrase Conjecture~\ref{conj:sLRC} as follows: Let~$\eps > 0$ be such that for a given set of pairwise distinct
velocities $v_1,\ldots,v_d > 0$ and every $s=(s_1,\dotsc,s_d)\in\R^d$, there is $t\in\R$ such that 
\[
\eps\leq\set{s_j+tv_j}\leq1-\eps, \quad \text{ for all } 1 \leq j \leq d.
\]
Here and in the sequel, we use the standard notation $\set{x} = x - \lfloor x \rfloor$ for the fractional part of a real number~$x \in \R$.
Now, we need to show that the maximal possible $\eps > 0$, such that the above property is satisfied for every set of pairwise distinct positive velocities, is equal to $\frac{1}{d+1}$.

In other words, the distance from $s_j+tv_j$ to a half-integer is at most $\frac{1}{2}-\eps$. Denoting $v=(v_1,\dotsc,v_d)$, we can further rephrase the above problem as follows:
Suppose that for every $s\in\R^d$, the line $s+\R v$ has $\ell_\infty$-distance $\leq\frac{1}{2}-\eps$ from the shifted lattice
\[\Z^d+(\tfrac{1}{2},\dotsc,\tfrac{1}{2}),\]
or equivalently, the point $(\tfrac{1}{2},\dotsc,\tfrac{1}{2})$ has $\ell_\infty$-distance $\leq\frac{1}{2}-\eps$ from the lattice arrangement of lines $s+\Z^d+\R v$.
It is thus crucial to describe the (closure of) the set $\Z^d+\R v$, or equivalently, the line $\R v$ modulo the integer lattice. 
This description was given in \cite[Lem.~2.3]{henzemalikiosis}:

\begin{lemma}
\label{latticelines}
The closure of the lattice arrangement of lines $\Z^d+\R v$ equals
\[
\mathcal{E}_v : =\set{ \xi \in \R^d : \ell^\intercal \xi \in \Z, \;\textrm{for all } \ell \in \Z^d \cap v^{\perp}},
\]
where $v^\perp$ denotes the orthogonal complement of the linear hull of~$v$.
\end{lemma}

Now, we are ready to show the reduction to integer velocities.
For this purpose, let $\eta$ denote the $\ell_\infty$-distance from the point $(\tfrac{1}{2},\dotsc,\tfrac{1}{2})$ to the lattice arrangement
of lines $s+\Z^d+\R v$, which by Lemma \ref{latticelines} equals the $\ell_\infty$-distance from the point $(\tfrac{1}{2},\dotsc,\tfrac{1}{2})$ to $s+\mathcal{E}_v$. 
It suffices to find an integer vector $y=(y_1,\dotsc,y_d)$ with pairwise distinct positive coordinates, such that the distance from $(\tfrac{1}{2},\dotsc,\tfrac{1}{2})$ to
$s+\mathcal{E}_y$, denoted by $\eta'$, satisfies $\eta\leq\eta'$. 

Consider the lattice $\La_v=\Z^d\cap v^\perp$, which has rank $r\leq d-1$. Therefore, the set of integer vectors orthogonal to all elements of $\La_v$ forms a lattice of rank $d-r\geq1$, say $\Ga$. Consider the
subspace $V$ of $\R^d$ spanned by $\Ga$, so that $v\in V$. The vector $v$ also belongs to the intersection of $V$ with the positive orthant $\R_{>0}^d$; since the latter is nonempty, it must be a convex cone
with apex at the origin and dimension $d-r$, hence the dimension of the linear hull of $\Ga^+:=\Ga\cap \R_{>0}^d$ must also be exactly $d-r$. Moreover, since $v$ has pairwise distinct positive coordinates, 
it does not belong to any hyperplane~$H_{ij}$ with defining equation $x_i=x_j$, for any $1\leq i<j\leq d$.
This shows that the intersection $V_{ij}=H_{ij}\cap V$ has dimension exactly $d-r-1$, hence the sublattice $\Ga_{ij}:=\Ga\cap H_{ij}$ has rank at most $d-r-1$. Thus, the set
\[
\Ga^+\setminus \Bigg( \bigcup_{1\leq i<j\leq d}\Ga_{ij} \Bigg) \subseteq \Z^d \cap \R^d_{>0}
\]
is nonempty, and we take an arbitrary element in this set, say $y$.

We show that $y$ satisfies the desired properties.
First of all, $y$ has positive integer coordinates, which are also pairwise distinct since $y\notin H_{ij}$ by definition, for all
$1\leq i<j\leq d$.
We prove that $\mathcal{E}_y\subseteq\mathcal{E}_v$, thus finishing the desired reduction.

Let $\xi\in\mathcal{E}_y$ be arbitrary, so that $\ell^\intercal \xi \in \Z$ for all $\ell \in \Z^d \cap y^\perp$ in view of Lemma~\ref{latticelines}.
We recall that $y\in\Ga$, and every element of $\Ga$ is orthogonal to every element of $\La_v$, therefore $\La_v\subseteq\Z^d\cap y^\perp$. This shows that $\ell^\intercal \xi \in \Z$ also holds for all $\ell\in\La_v$, which eventually yields $\xi\in\mathcal{E}_v$, again by Lemma~\ref{latticelines}.
As~$\xi$ was arbitrary, we obtain $\mathcal{E}_y\subseteq\mathcal{E}_v$ as desired, hence $\eta\leq\eta'$.

\subsection{Zonotopes associated with the Lonely Runner Problem}
\label{sect:zonotopal-lrc}

We use the geometric interpretation of the Lonely Runner Problem established in~\cite{henzemalikiosis} and~\cite{beckhostenschymura2018lrpolyhedra} in order to reduce the case $d=3$ of \Cref{conj:sLRC} to finitely many instances, which we then separately tackle by the covering radius algorithm developed in Section~\ref{sect:novel-algorithm}.

Czerwi\'{n}ski \& Grytczuk~\cite[Thm.~6]{czerwinskigrytczuk2008invisible} describe the time $t$ at which the maximal distance~$\lambda_v$ of the runners from the stagnant runner is attained, for the case $s=(s_1,\ldots,s_d) = \zero$.
Their arguments work for every \emph{fixed} starting configuration $s \in \R^d$ and they imply an algorithm with time complexity $\O(d^2 \cdot v_{max})$, where $v_{max} = \max_{1 \leq i \leq d} v_i$, to compute this distance~$\lambda_v$.
One can thus computationally check Conjecture~\ref{conj:sLRC} for a given velocity vector~$v$ and an arbitrary but fixed starting point~$s$.
For the full statement of Conjecture~\ref{conj:sLRC} we however need to do this computation a priori for \emph{every} starting configuration of the runners.
Here, the covering radius comes into play.

We now describe the essential parts on how to transform Conjecture~\ref{conj:sLRC} into a problem on bounding the covering radius of certain lattice zonotopes.
For full details on the geometric arguments used in this section, as well as the reduction of the Lonely Runner Conjecture and similar versions thereof for finding (or proving the existence of) lattice points inside certain zonotopes, we refer the reader to \cite[Sect.~2.2 \& 2.3]{henzemalikiosis} and~\cite{beckhostenschymura2018lrpolyhedra}.

As described after the statement of~\cite[Conj.~10]{beckhostenschymura2018lrpolyhedra}, or in~\cite[Thm.~1.1]{henzemalikiosis}, the following holds for an arbitrary number of runners~$d$:
With every set of integer speeds $0 < v_1 < \ldots < v_d$, with $\gcd(v_1,\ldots,v_d) = 1$, we associate a lattice zonotope $Z_v$ in $\R^{d-1}$ generated by~$d$ lattice vectors $u_1,\ldots,u_d \in \Z^{d-1}$ in \emph{general linear position}, which means that every $d-1$ of them are linearly independent.
More precisely,
\[
Z_v := \sum_{j=1}^d [\zero,u_j] \subseteq \R^{d-1}.
\]
The connection between the generators $u_1,\ldots,u_d$ and the \emph{velocity vector} $v=(v_1,\ldots,v_d)$ is described in detail in \cite{henzemalikiosis}, in the discussion after Lemma~3.1 therein.
The main points are as follows:
Since $v \in \Z^d_{>0}$, the intersection $\La_v=\Z^d \cap v^{\perp}$ is a sublattice of~$\Z^d$ of dimension $d-1$.
Consider any $(d-1)\times d$ matrix $A$, whose rows $a_1,\ldots,a_{d-1}$ constitute a basis of the lattice~$\La_v$.
The associated lattice zonotope $Z_v$ is then generated by the 
\emph{columns} $u_1,\ldots,u_d$ of $A$.\footnote{This zonotope is not, of course, unique, but all such zonotopes are unimodularly equivalent and thus have the same properties with respect to the volume and the covering radius, described afterwards.}
 
One important fact concerning the velocities~$v_i$ and the zonotope $Z_v$ is the following:
For each $i \in [d]$, the volume of the parallelepiped $\sum_{j=1,j \neq i}^d [\zero,u_j]$, spanned by all but one of the~$d$ generators of~$Z_v$, equals~$v_i$.
In particular,
\begin{equation}\label{zonovolume}
 \vol(Z_v) = v_1 + v_2 + \ldots + v_d.
\end{equation}
Now, it can be inferred from~\cite[Sect.~5]{henzemalikiosis}, that the~$d$ runners with velocities $v_1,\ldots,v_d$ satisfy \Cref{conj:sLRC} if and only if
\begin{align}
\mu(Z_v) \leq \frac{d-1}{d+1}.\label{eq:sLRC-zonotopal}
\end{align}
This means that after contracting~$Z_v$ by the factor $\frac{d-1}{d+1}$, we may shift it anywhere in~$\R^{d-1}$ and always find some lattice point of~$\Z^{d-1}$ that is contained in the shifted copy.
The originial Lonely Runner Conjecture corresponds to the question whether a \emph{certain} position of $\frac{d-1}{d+1} Z_v$ contains a lattice point.
For this reason, we may call \Cref{conj:sLRC} the \emph{Shifted Lonely Runner Conjecture} in the sequel.

Before we describe how this geometric point of view can be used to reduce Conjecture~\ref{conj:sLRC} for $d=3$ to only a few explicitly given velocity vectors, we remark that it suffices to consider three speeds that are \emph{pairwise} coprime.

\begin{proposition}
\label{prop:pairwise-coprime-d3}
Let $v=(v_1,v_2,v_3) \in \Z^3_{>0}$ with $\gcd(v_1,v_2,v_3)=1$ be a velocity vector with pairwise distinct coordinates, and let $s \in \R^3$ be a triple of starting points.
Suppose further that $\gcd(v_1,v_2)=\ell>1$.
Then, there exists some $t\in\R$ such that $\frac14 < \set{s_j+v_jt} < \frac34$, for every $j = 1,2,3$.
\end{proposition}

\begin{proof}
Because \Cref{conj:sLRC} holds for $d=2$ (i.e.~for a stagnant runner and two runners with distinct integer non-zero speeds $v_1$ and $v_2$ and arbitrary starting points $s_1$ and $s_2$), 
there is some $t_0\in\R$ such that $\frac13 \leq \set{s_j+v_jt_0} \leq \frac23$, for $j=1, 2$.
Since $\ell$ divides $v_1$ and $v_2$, for any
\[
t'\in\set{t_0,t_0+\frac{1}{\ell},\ldots,t_0+\frac{\ell-1}{\ell}},
\]
the fractional parts $\set{s_j+v_jt'}$ and $\set{s_j+v_jt_0}$ agree, for $j=1,2$.
Furthermore, for at least one of the values for~$t'$ we must have $\frac14 \leq \set{s_3+v_3t'} \leq \frac34$, since $\gcd(\ell,v_3)=1$.
Indeed, consider the following complex numbers of modulus one:
\[z_k=e^{2\pi i(s_3+v_3t_0+v_3\frac{k}{\ell})}, \;\;\;\; 0\leq k\leq \ell-1.\]
The inequality $\frac14 \leq \set{s_3+v_3t_0+v_3\frac{k}{\ell}} \leq \frac34$ is equivalent to $\mathrm{Re}(z_k) \leq 0$.
If the latter fails for every $k \in \{0,\ldots,\ell-1\}$, then the
real part of the sum of the $z_k$ must be positive. But this does not hold, as $\ell > 1$ and thus
\[
\sum_{k=0}^{\ell-1}z_k=e^{2\pi i(s_3+v_3t_0)}\sum_{k=0}^{\ell-1}e^{2\pi iv_3\frac{k}{\ell}}=0.
\]
Now, for this choice of~$t'$, we have $\frac14 < \frac13\leq \set{s_j+v_jt'}\leq\frac23 < \frac34$, for $j=1,2$, and therefore there is some $\varepsilon>0$ such that
 $\frac14 < \set{s_j+v_jt} < \frac34$ holds for all $t\in(t'-\varepsilon,t'+\varepsilon)$. Moreover, there is a half-interval having~$t'$ as an endpoint, i.e.~of the form $(t',t'+\delta)$ or $(t'-\delta,t')$ for some $\delta>0$, such that 
$\frac14 < \set{s_3+v_3t} < \frac34$ holds as well, for every $t$ in this interval. 
We conclude that there is some $t$ close to $t'$ (in particular, having distance at most $\max(\delta,\varepsilon)$)
such that $\frac14 < \set{s_j+v_jt} < \frac34$ holds for every $j = 1,2,3$, as desired.
\end{proof}

In view of Proposition \ref{prop:pairwise-coprime-d3} we are left with the case when the speeds $v_1,v_2,v_3$ are pairwise coprime.
In this case, we can use the extended Euclidean algorithm to write $1$ as an integer linear combination of $v_1$ and $v_2$. Using this representation, we can represent $v_3$ as
\[
v_3=\kappa v_1+\la v_2, \ \text{ for some }	\kappa, \la\in v_3\Z.
\]
The associated zonotope~$Z_v$ is then generated by the columns of the following $2\times 3$ matrix:
\begin{equation}\label{LRzonotope}
\begin{pmatrix}
 v_2 & -v_1 & 0\\
 \kappa & \la & -1
\end{pmatrix}.
\end{equation}
To verify this, we observe first that both rows are orthogonal to the velocity vector $v=(v_1,v_2,v_3)$ and are linearly independent. Then, if we add the row
$(\kappa/v_3, \la/v_3, 0)$ to the above matrix, we obtain a matrix with determinant~$1$, thus showing that the two rows of \eqref{LRzonotope} generate $\Z^3\cap v^\perp$.
It is also worthwhile mentioning that the absolute values of the $2\times2$ determinants in \eqref{LRzonotope} are precisely $v_1, v_2, v_3$.


\subsection{A geometric reduction for the case of three runners with non-zero velocities}
\label{sect:geometric-reduction-three-runners}

Throughout the following let $v=(v_1,v_2,v_3) \in \Z_{>0}^3$ be a velocity vector with pairwise distinct speeds, and let $Z_v=\sum_{i=1}^3[\zero,u_i]$ be the associated planar lattice zonotope as defined above.
In view of~\eqref{zonovolume}, we have $\vol(Z_v)=v_1+v_2+v_3$.
Since $Z_v$ has only integral vertices, its lattice-width $w(Z_v)$ is an integer.
Moreover, as the generators of $Z_v$ are in general linear position, there is at least one lattice point that is interior to $Z_v$.
In fact, since~$u_3$ is not parallel to~$u_1$ or~$u_2$, adding the segment $[\zero,u_3]$ to the parallelogram $Q := [\zero,u_1]+[\zero,u_2]$ turns one of the four vertices of~$Q$ into an interior point of~$Z_v$.
Such an interior lattice point implies $w(Z_v) \geq 2$.

By virtue of~\eqref{eq:sLRC-zonotopal}, for proving the Shifted Lonely Runner Conjecture for~$v$, we need to show that $\mu(Z_v) \leq \frac12$.
To get a first reduction, we apply the flatness theorem~\eqref{eq:flatness-thm} in the plane for centrally symmetric bodies, with the best possible constant, which is~$2$.
In fact, it follows from~\cite[Cor.~2.7]{averkovwagner2012inequalities} that
\begin{align}
\mu(Z_v) \leq 2 \cdot w(Z_v)^{-1},\label{eqn:planar-flatness}
\end{align}
because $Z_v$ as a zonotope is centrally symmetric.\footnote{This inequality, and an extension to larger dimensions, has been claimed in~\cite[Thm.~(2.13)]{kannanlovasz1988covering}, but the proof has an issue that to the best of our knowledge was not fixed as of today. A different and valid proof for the case of dimension two is provided in~\cite{averkovwagner2012inequalities}.}
Moreover, for the inequality to be tight it is necessary for~$Z_v$ to be a parallelogram.
However, as~$Z_v$ is a hexagon, this means we actually have strict inequality.
Therefore, if $w(Z_v) \geq 4$, then $\mu(Z_v) < \frac12$, and thus we may assume that $w(Z_v) \in \{2,3\}$ in the sequel.

We first show that the case of lattice-width $w(Z_v) = 2$ cannot occur.

\begin{lemma}
\label{lem:no-lattice-width-two-zonotopes}
Let $v \in \Z^3_{>0}$ be a velocity vector with pairwise distinct entries.
Then, the planar lattice zonotope $Z_v$ has lattice-width $w(Z_v) \geq 3$.
\end{lemma}

\begin{proof}
We argue by contradiction and assume that $w(Z_v) = 2$.
The zonotope~$Z_v$ decomposes into (suitable lattice translates of) the three lattice parallelograms $P_j = \sum_{i=1, i\neq j}^3 [\zero,u_i]$, for $j=1,2,3$.
Since $w(Z_v)=2$ and $Z_v$ is a hexagon, there must be a direction $z \in \Z^2 \setminus \{\zero\}$ such that $w(Z_v) = \max_{x \in Z_v} x^\intercal z - \min_{x \in Z_v} x^\intercal z = 2$ and $z$ is orthogonal to an edge (and thus to a generator) of~$Z_v$.
Without loss of generality, we assume that~$z$ is orthogonal to~$u_1$.
This implies that $w(P_2) = w(P_3) = 1$ and that all the lattice points contained in~$P_2$ and $P_3$ are distributed on the edges parallel to~$u_1$.
Compare with the illustration in Figure~\ref{fig:zonotope-width}.
Since the $P_j$ are lattice parallelograms this means that $P_2$ and $P_3$ contain exactly the same number of lattice points, and thus they have the same area.
This means however that $v_2 = \vol(P_2) = \vol(P_3) = v_3$ (see the discussion before~\eqref{zonovolume}), contradicting the assumption that the entries of~$v$ are pairwise distinct.
\end{proof}

\begin{figure}[ht]
\includegraphics[scale=1]{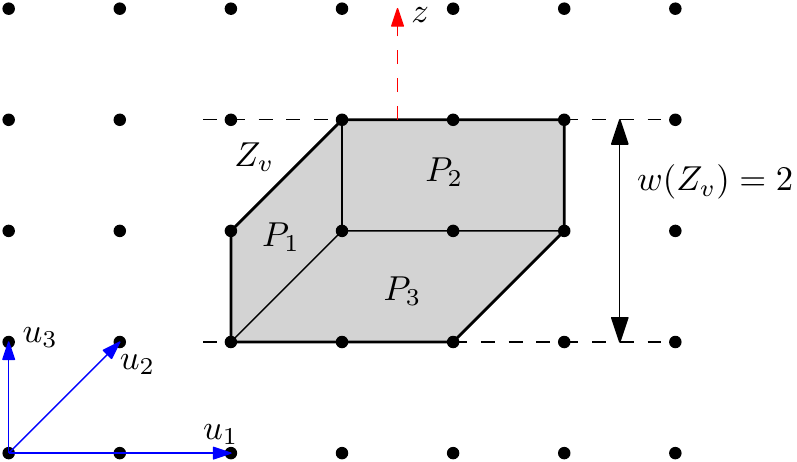}
\caption{If $w(Z_v) = 2$, then the areas of the parallelograms $P_1$, $P_2$, $P_3$ cannot be pairwise distinct.}
\label{fig:zonotope-width}
\end{figure}

Let us now consider the case that $w(Z_v)=3$.
By definition of the lattice-width, there is a translation that puts the dilate $\frac{1}{w(Z_v)} Z_v$ in between two parallel neighboring lattice lines.
In particular, $\frac{1}{w(Z_v)} Z_v$ has a lattice-free translate and thus $\mu(Z_v)w(Z_v) \geq 1$.
If $1=\mu(Z_v)w(Z_v)=3\mu(Z_v)$, then $\mu(Z_v) = \frac13 < \frac12$ and so we can safely assume that $\mu(Z_v)w(Z_v) > 1$.
Under this condition, Averkov \& Wagner~\cite[Cor.~2.7]{averkovwagner2012inequalities} proved that
\begin{align}
\vol(Z_v) \leq \frac{w(Z_v)^2}{2\mu(Z_v)w(Z_v)-2} = \frac{9}{6\mu(Z_v)-2}.\label{eqn:averkov-wagner}
\end{align}
Rearranging terms yields the inequality
\[
\mu(Z_v) \leq \frac{3}{2 \vol(Z_v)} + \frac13 = \frac{3}{2 (v_1+v_2+v_3)} + \frac13,
\]
and thus $\mu(Z_v) < \frac12$ as soon as $v_1+v_2+v_3 \geq 10$.

In view of the investigations above and Proposition~\ref{prop:pairwise-coprime-d3}, we only need to check the triples $v=(v_1,v_2,v_3) \in \Z^3_{>0}$ 
with pairwise relatively prime entries and $v_1+v_2+v_3 \leq 9$ for establishing the Shifted Lonely Runner Conjecture for three non-zero runners, 
and characterize the extremal velocity vectors along as well.
In particular, these are only the four triples:
\[
(1,2,3),
(1,2,5),
(1,3,4),(1,3,5).
\]
%

We implemented our algorithm $\covradname{P}$ in \texttt{sage}~\cite{sagemath} and computed the covering radius of the four remaining cases, with the results gathered in Table~\ref{tbl:computer-examples-shortlist}.
The source code of our implementation can be found at \url{https://github.com/mschymura/covering-radius-computation}.

\begin{table}[ht]
\begin{tabular}{|l|c|c|}
\hline
$(v_1,v_2,v_3)$ & generators for $Z_v$ & $\mu(Z_v)$ \\ \hline\hline
\phantom{.}$(1,2,3)$ & $\begin{pmatrix} 2 & -1 & 0\\ 3 & 0 & -1 \end{pmatrix}$ & $\frac{1}{2}$ \\ \hline
\phantom{.}$(1,2,5)$ & $\begin{pmatrix} 2 & -1 & 0\\ 5 & 0 & -1 \end{pmatrix}$ & $\frac{3}{7}$ \\ \hline
\phantom{.}$(1,3,4)$ & $\begin{pmatrix} 3 & -1 & 0\\ 4 & 0 & -1 \end{pmatrix}$ & $\frac{3}{7}$ \\ \hline
\phantom{.}$(1,3,5)$ & $\begin{pmatrix} 3 & -1 & 0\\ 5 & 0 & -1 \end{pmatrix}$ & $\frac{3}{8}$ \\ \hline
\end{tabular}
\caption{The covering radius of the zonotopes~$Z_v$ corresponding to the four remaining triples above.}
\label{tbl:computer-examples-shortlist}
\end{table}

Of course, the validity of the bound $\mu(Z_v) \leq \frac12$ could also be checked by hand for these very special cases.
The lattice-width of the zonotope corresponding to the first three examples equals~$3$, whereas for the case $v=(1,3,5)$ we have $w(Z_v) = 4$ and thus $\mu(Z_v) < \frac12$ holds by the discussion after~\eqref{eqn:planar-flatness}.
We chose to compute the exact covering radii for informative reasons.
In summary, we have proven:

\begin{theorem}
\label{thm:sLRC-dim3}
\Cref{conj:sLRC} holds for $d=3$ and the only extremal velocity vector (up to scaling and permuting coordinates) is $v=(1,2,3)$.
\end{theorem}

\begin{remark}
The attempt to reduce the Lonely Runner Problem to a finite list of velocity vectors is not new.
In fact, Tao~\cite{tao2018someremarks} obtained such a reduction for an arbitrary number of runners, in the original setting of \Cref{LRC} with common starting points.
He shows that for a fixed~$d$ it suffices to look at integer velocities $v_1,\ldots,v_d$ with $|v_i| \in d^{\O(d^2)}$, for $1 \leq i \leq d$.
However, this leaves too large a number of cases to check in order to be practically applicable to solving the conjecture, even for a small dimension.
\end{remark}

It is intriguing to try and extend our arguments for Theorem~\ref{thm:sLRC-dim3} towards a reduction for an arbitrary number of runners.
However, the success of the case distinction between the possible values of~$w(Z_v)$, via the planar flatness theorem~\eqref{eqn:planar-flatness} and the Averkov-Wagner bound~\eqref{eqn:averkov-wagner}, seems to be limited to the case $d=3$.
On the one hand, equally precise inequalities between the lattice-width, the covering radius and the volume of (centrally symmetric) polytopes are not available to date.
For instance, the currently best-known bound for the flatness theorem for a convex body~$K \subseteq \R^3$ can be found in~\cite[Thm.~3.2]{dashdobbsgunluknowickiswirszcz2014latticefree} and reads $\mu(K) w(K) \leq 1 + \frac{2}{\sqrt{3}}+(\frac{90}{\pi^2})^{1/3} \approx 4.2439$.
Thus, the required bound $\mu(Z_v) \leq \frac35$ for $d=4$ runners with velocity vector $v \in \Z^4_{>0}$ would be satisfied only if $w(Z_v) \geq 8$.
On the other hand, the associated zonotope~$Z_{\bar v}$ of the velocity vector ${\bar v} = (1,2,\ldots,d)$ is generated by the columns of the following matrix (we only show its non-zero entries):
\[
\left(
\begin{array}{cccccc}
     2 & -1 &       &  &  &  \\
     1 &  1 &     -1 & &  &  \\
\vdots &    & \ddots & \ddots &  & \\
     1 &    &        &      1 &     -1 & \\
     1 &    &  &       &      1 & -1
\end{array}
\right).
\]
The first row of this matrix shows that the lattice-width of~$Z_{\bar v}$ in direction $e_1=(1,0,\ldots,0)^\intercal$ equals~$3$, for any $d \geq 3$.
Thus, already for $d=4$ runners we would need to deal with the cases $w(Z_v) \in \{3,4,5,6,7\}$ separately, requiring substantial new ideas to complete the desired reduction.

\subsection{More computer experiments and Kravitz' Loneliness Spectrum Conjecture}

In order to collect some more computational data on covering radii, we computed this parameter for the zonotopes~$Z_v$ for all triples $v = (v_1,v_2,v_3) \in \Z^3_{>0}$ with coprime entries satisfying $v_1+v_2+v_3 \leq 18$.
The results are summarized in Tables~\ref{tbl:computer-examples-18-pt1}, \ref{tbl:computer-examples-18-pt2} and~\ref{tbl:computer-examples-18-pt3}.

Let $\ML(v_1,\ldots,v_d)$ denote the \emph{maximum loneliness} of the runners with velocities $v_1,\ldots,v_d$, defined as the maximum distance from an integer that all runners can attain simultaneously.
Kravitz' Loneliness Spectrum Conjecture~\cite[Conj.~1.2]{kravitz2019barely} states that for positive coprime integer velocities $v_1,\ldots,v_d$ either
\[
\ML(v_1,\ldots,v_d) = \frac{m}{dm+1},\quad\textrm{for some } m\in\N,\quad \textrm{ or }\quad \ML(v_1,\ldots,v_d) \geq \frac{1}{d}.
\]
Note that this is assuming that all runners start at the same place, that is, $s_1 = \ldots = s_d = 0$.
Translating this conjecture to the shifted setting and the covering radius formulation (cf.~\cite[Sect.~3]{henzemalikiosis}) is equivalent to saying that either
\[
\mu(Z_v) = \frac{(d-2)m+1}{dm+1},\quad\textrm{for some } m\in\N,\quad \textrm{ or }\quad \mu(Z_v) \leq \frac{d-2}{d}.
\]
For the three runner case that we investigated above, this means that
\[
\mu(Z_v) = \frac{m+1}{3m+1},\quad\text{for some }m\in\N,\quad \textrm{ or }\quad \mu(Z_v) \leq \frac{1}{3}.
\]
Inspecting the computed covering radii in Tables~\ref{tbl:computer-examples-18-pt1}, \ref{tbl:computer-examples-18-pt2} and~\ref{tbl:computer-examples-18-pt3}, one finds that this holds true for almost all the values.
The only exceptional covering radii are $\frac{4}{9},\frac{5}{12},\frac{9}{23},\frac{15}{41}$.
Thus, Kravitz' Conjecture does not hold unmodified in the shifted setting.

However, based on this limited data we may extend the Loneliness Spectrum Conjecture (at least for three runners) to the shifted setting as follows:
\[
\mu(Z_v) = \frac{m+1}{3m+j},\ \text{for some }m\in\N\text{ and }j\in\{-1,0,1\},\quad \textrm{ or }\quad \mu(Z_v) \leq \frac{1}{3}.
\]


\subsection*{Acknowledgments}
We thank J\"org M.~Wills and Gennadiy Averkov for thorough reading of an earlier version of the manuscript, and for providing valuable comments and suggestions.
We thank the anonymous referees for very careful reading and for suggestions that improved the quality of the presentation of our material.

\bibliographystyle{amsplain}
\bibliography{mybib}

\appendix

\begin{table}[ht]
\begin{tabular}{|c|c|c||c|c|c|}
\hline
$(v_1,v_2,v_3)$ & generators for $Z_v$ & $\mu(Z_v)$ & $(v_1,v_2,v_3)$ & generators for $Z_v$ & $\mu(Z_v)$ \\ \hline\hline
\phantom{.}$(1,2,3)$ & $\begin{pmatrix} 2 & -1 & 0\\ 1 & 1 & -1 \end{pmatrix}$ & $\frac{1}{2}$ &
\phantom{.}$(1,2,4)$ & $\begin{pmatrix} 2 & -1 & 0\\ 0 & 2 & -1 \end{pmatrix}$ & $\frac{3}{7}$\\ \hline
\phantom{.}$(1,2,5)$ & $\begin{pmatrix} 2 & -1 & 0\\ 1 & 2 & -1 \end{pmatrix}$ & $\frac{3}{7}$ &
\phantom{.}$(1,2,6)$ & $\begin{pmatrix} 2 & -1 & 0\\ 2 & 2 & -1 \end{pmatrix}$ & $\frac{3}{7}$ \\ \hline
\phantom{.}$(1,2,7)$ & $\begin{pmatrix} 2 & -1 & 0\\ 1 & 3 & -1 \end{pmatrix}$ & $\frac{9}{23}$ &
\phantom{.}$(1,2,8)$ & $\begin{pmatrix} 2 & -1 & 0\\ 2 & 3 & -1 \end{pmatrix}$ & $\frac{2}{5}$\\ \hline
\phantom{.}$(1,2,9)$ & $\begin{pmatrix} 2 & -1 & 0\\ 3 & 3 & -1 \end{pmatrix}$ & $\frac{2}{5}$ &
\phantom{.}$(1,2,10)$ & $\begin{pmatrix} 2 & -1 & 0\\ 2 & 4 & -1 \end{pmatrix}$ & $\frac{3}{8}$ \\ \hline
\phantom{.}$(1,2,11)$ & $\begin{pmatrix} 2 & -1 & 0\\ 3 & 4 & -1 \end{pmatrix}$ & $\frac{5}{13}$ &
\phantom{.}$(1,2,12)$ & $\begin{pmatrix} 2 & -1 & 0\\ 4 & 4 & -1 \end{pmatrix}$ & $\frac{5}{13}$\\ \hline
\phantom{.}$(1,2,13)$ & $\begin{pmatrix} 2 & -1 & 0\\ 3 & 5 & -1 \end{pmatrix}$ & $\frac{15}{41}$ &
\phantom{.}$(1,2,14)$ & $\begin{pmatrix} 2 & -1 & 0\\ 4 & 5 & -1 \end{pmatrix}$ & $\frac{3}{8}$\\ \hline
\phantom{.}$(1,2,15)$ & $\begin{pmatrix} 2 & -1 & 0\\ 5 & 5 & -1 \end{pmatrix}$ & $\frac{3}{8}$ &
\phantom{.}$(1,3,4)$ & $\begin{pmatrix} 3 & -1 & 0\\ 1 & 1 & -1 \end{pmatrix}$ & $\frac{3}{7}$ \\ \hline
\phantom{.}$(1,3,5)$ & $\begin{pmatrix} 3 & -1 & 0\\ 2 & 1 & -1 \end{pmatrix}$ & $\frac{3}{8}$ &
\phantom{.}$(1,3,6)$ & $\begin{pmatrix} 3 & -1 & 0\\ 0 & 2 & -1 \end{pmatrix}$ & $\frac{1}{3}$ \\ \hline
\phantom{.}$(1,3,7)$ & $\begin{pmatrix} 3 & -1 & 0\\ 1 & 2 & -1 \end{pmatrix}$ & $\frac{3}{8}$ &
\phantom{.}$(1,3,8)$ & $\begin{pmatrix} 3 & -1 & 0\\ 2 & 2 & -1 \end{pmatrix}$ & $\frac{1}{3}$ \\ \hline
\phantom{.}$(1,3,9)$ & $\begin{pmatrix} 3 & -1 & 0\\ 0 & 3 & -1 \end{pmatrix}$ & $\frac{4}{13}$ &
\phantom{.}$(1,3,10)$ & $\begin{pmatrix} 3 & -1 & 0\\ 1 & 3 & -1 \end{pmatrix}$ & $\frac{4}{13}$ \\ \hline
\phantom{.}$(1,3,11)$ & $\begin{pmatrix} 3 & -1 & 0\\ 2 & 3 & -1 \end{pmatrix}$ & $\frac{1}{3}$ &
\phantom{.}$(1,3,12)$ & $\begin{pmatrix} 3 & -1 & 0\\ 3 & 3 & -1 \end{pmatrix}$ & $\frac{4}{13}$\\ \hline
\phantom{.}$(1,3,13)$ & $\begin{pmatrix} 3 & -1 & 0\\ 1 & 4 & -1 \end{pmatrix}$ & $\frac{16}{55}$ &
\phantom{.}$(1,3,14)$ & $\begin{pmatrix} 3 & -1 & 0\\ 2 & 4 & -1 \end{pmatrix}$ & $\frac{5}{17}$ \\ \hline
\end{tabular}
\caption{The covering radius of the hexagons corresponding to the triples of velocities which sum to at most $18$ - Part I.}
\label{tbl:computer-examples-18-pt1}
\end{table}

\begin{table}[ht]
\begin{tabular}{|c|c|c||c|c|c|}
\hline
$(v_1,v_2,v_3)$ & generators for $Z_v$ & $\mu(Z_v)$ & $(v_1,v_2,v_3)$ & generators for $Z_v$ & $\mu(Z_v)$ \\ \hline\hline
\phantom{.}$(1,4,5)$ & $\begin{pmatrix} 4 & -1 & 0\\ 1 & 1 & -1 \end{pmatrix}$ & $\frac{4}{9}$ &
\phantom{.}$(1,4,6)$ & $\begin{pmatrix} 4 & -1 & 0\\ 2 & 1 & -1 \end{pmatrix}$ & $\frac{2}{5}$ \\ \hline
\phantom{.}$(1,4,7)$ & $\begin{pmatrix} 4 & -1 & 0\\ -1 & 2 & -1 \end{pmatrix}$ & $\frac{4}{11}$ &
\phantom{.}$(1,4,8)$ & $\begin{pmatrix} 4 & -1 & 0\\ 0 & 2 & -1 \end{pmatrix}$ & $\frac{1}{3}$ \\ \hline
\phantom{.}$(1,4,9)$ & $\begin{pmatrix} 4 & -1 & 0\\ 1 & 2 & -1 \end{pmatrix}$ & $\frac{4}{13}$ &
\phantom{.}$(1,4,10)$ & $\begin{pmatrix} 4 & -1 & 0\\ 2 & 2 & -1 \end{pmatrix}$ & $\frac{2}{7}$ \\ \hline
\phantom{.}$(1,4,11)$ & $\begin{pmatrix} 4 & -1 & 0\\ -1 & 3 & -1 \end{pmatrix}$ & $\frac{4}{15}$ &
\phantom{.}$(1,4,12)$ & $\begin{pmatrix} 4 & -1 & 0\\ 0 & 3 & -1 \end{pmatrix}$ & $\frac{1}{4}$ \\ \hline
\phantom{.}$(1,4,13)$ & $\begin{pmatrix} 4 & -1 & 0\\ 1 & 3 & -1 \end{pmatrix}$ & $\frac{2}{7}$ &
\phantom{.}$(1,5,6)$ & $\begin{pmatrix} 5 & -1 & 0\\ 1 & 1 & -1 \end{pmatrix}$ & $\frac{3}{7}$ \\ \hline
\phantom{.}$(1,5,7)$ & $\begin{pmatrix} 5 & -1 & 0\\ 2 & 1 & -1 \end{pmatrix}$ & $\frac{3}{8}$ &
\phantom{.}$(1,5,8)$ & $\begin{pmatrix} 5 & -1 & 0\\ 3 & 1 & -1 \end{pmatrix}$ & $\frac{18}{53}$ \\ \hline
\phantom{.}$(1,5,9)$ & $\begin{pmatrix} 5 & -1 & 0\\ -1 & 2 & -1 \end{pmatrix}$ & $\frac{1}{3}$ &
\phantom{.}$(1,5,10)$ & $\begin{pmatrix} 5 & -1 & 0\\ 0 & 2 & -1 \end{pmatrix}$ & $\frac{1}{3}$\\ \hline
\phantom{.}$(1,5,11)$ & $\begin{pmatrix} 5 & -1 & 0\\ 1 & 2 & -1 \end{pmatrix}$ & $\frac{5}{16}$ &
\phantom{.}$(1,5,12)$ & $\begin{pmatrix} 5 & -1 & 0\\ 2 & 2 & -1 \end{pmatrix}$ & $\frac{5}{17}$ \\ \hline
\phantom{.}$(1,6,7)$ & $\begin{pmatrix} 6 & -1 & 0\\ 1 & 1 & -1 \end{pmatrix}$ & $\frac{5}{13}$ &
\phantom{.}$(1,6,8)$ & $\begin{pmatrix} 6 & -1 & 0\\ 2 & 1 & -1 \end{pmatrix}$ & $\frac{1}{3}$ \\ \hline
\phantom{.}$(1,6,9)$ & $\begin{pmatrix} 6 & -1 & 0\\ 3 & 1 & -1 \end{pmatrix}$ & $\frac{7}{23}$ &
\phantom{.}$(1,6,10)$ & $\begin{pmatrix} 6 & -1 & 0\\ -2 & 2 & -1 \end{pmatrix}$ & $\frac{11}{38}$ \\ \hline
\phantom{.}$(1,6,11)$ & $\begin{pmatrix} 6 & -1 & 0\\ -1 & 2 & -1 \end{pmatrix}$ & $\frac{5}{17}$ &
\phantom{.}$(1,7,8)$ & $\begin{pmatrix} 7 & -1 & 0\\ 1 & 1 & -1 \end{pmatrix}$ & $\frac{2}{5}$ \\ \hline
\phantom{.}$(1,7,9)$ & $\begin{pmatrix} 7 & -1 & 0\\ 2 & 1 & -1 \end{pmatrix}$ & $\frac{5}{16}$ &
\phantom{.}$(1,7,10)$ & $\begin{pmatrix} 7 & -1 & 0\\ 3 & 1 & -1 \end{pmatrix}$ & $\frac{8}{29}$ \\ \hline
\phantom{.}$(1,8,9)$ & $\begin{pmatrix} 8 & -1 & 0\\ 1 & 1 & -1 \end{pmatrix}$ & $\frac{2}{5}$ & & &\\ \hline
\end{tabular}
\caption{The covering radius of the hexagons corresponding to the triples of velocities which sum to at most $18$ - Part II.}
\label{tbl:computer-examples-18-pt2}
\end{table}

\begin{table}[ht]
\begin{tabular}{|c|c|c||c|c|c|}
\hline
$(v_1,v_2,v_3)$ & generators for $Z_v$ & $\mu(Z_v)$ & $(v_1,v_2,v_3)$ & generators for $Z_v$ & $\mu(Z_v)$ \\ \hline\hline
\phantom{.}$(2,3,4)$ & $\begin{pmatrix} 3 & -2 & 0\\ 2 & 0 & -1 \end{pmatrix}$ & $\frac{2}{5}$ & 
\phantom{.}$(2,3,5)$ & $\begin{pmatrix} 3 & -2 & 0\\ 1 & 1 & -1 \end{pmatrix}$ & $\frac{3}{7}$ \\ \hline
\phantom{.}$(2,3,6)$ & $\begin{pmatrix} 3 & -2 & 0\\ 0 & 2 & -1 \end{pmatrix}$ & $\frac{1}{3}$ &
\phantom{.}$(2,3,7)$ & $\begin{pmatrix} 3 & -2 & 0\\ 2 & 1 & -1 \end{pmatrix}$ & $\frac{1}{3}$ \\ \hline
\phantom{.}$(2,3,8)$ & $\begin{pmatrix} 3 & -2 & 0\\ 2 & 2 & -1 \end{pmatrix}$ & $\frac{4}{11}$ &
\phantom{.}$(2,3,9)$ & $\begin{pmatrix} 3 & -2 & 0\\ 0 & 3 & -1 \end{pmatrix}$ & $\frac{5}{17}$ \\ \hline
\phantom{.}$(2,3,10)$ & $\begin{pmatrix} 3 & -2 & 0\\ 4 & 2 & -1 \end{pmatrix}$ & $\frac{1}{3}$ &
\phantom{.}$(2,3,11)$ & $\begin{pmatrix} 3 & -2 & 0\\ 1 & 3 & -1 \end{pmatrix}$ & $\frac{2}{7}$ \\ \hline
\phantom{.}$(2,3,12)$ & $\begin{pmatrix} 3 & -2 & 0\\ 3 & 3 & -1 \end{pmatrix}$ & $\frac{2}{7}$ &
\phantom{.}$(2,3,13)$ & $\begin{pmatrix} 3 & -2 & 0\\ 2 & 3 & -1 \end{pmatrix}$ & $\frac{5}{16}$ \\ \hline
\phantom{.}$(2,4,5)$ & $\begin{pmatrix} 2 & -1 & 0\\ 1 & 2 & -2 \end{pmatrix}$ & $\frac{7}{19}$ &
\phantom{.}$(2,4,7)$ & $\begin{pmatrix} 2 & -1 & 0\\ 1 & 3 & -2 \end{pmatrix}$ & $\frac{9}{25}$ \\ \hline
\phantom{.}$(2,4,9)$ & $\begin{pmatrix} 2 & -1 & 0\\ 3 & 3 & -2 \end{pmatrix}$ & $\frac{4}{11}$ &
\phantom{.}$(2,4,11)$ & $\begin{pmatrix} 2 & -1 & 0\\ 3 & 4 & -2 \end{pmatrix}$ & $\frac{13}{37}$ \\ \hline
\phantom{.}$(2,5,6)$ & $\begin{pmatrix} 5 & -2 & 0\\ 3 & 0 & -1 \end{pmatrix}$ & $\frac{4}{13}$ &
\phantom{.}$(2,5,7)$ & $\begin{pmatrix} 5 & -2 & 0\\ 1 & 1 & -1 \end{pmatrix}$ & $\frac{5}{12}$ \\ \hline
\phantom{.}$(2,5,8)$ & $\begin{pmatrix} 5 & -2 & 0\\ -1 & 2 & -1 \end{pmatrix}$ & $\frac{4}{13}$ &
\phantom{.}$(2,5,9)$ & $\begin{pmatrix} 5 & -2 & 0\\ 2 & 1 & -1 \end{pmatrix}$ & $\frac{5}{14}$ \\ \hline
\phantom{.}$(2,5,10)$ & $\begin{pmatrix} 5 & -2 & 0\\ 0 & 2 & -1 \end{pmatrix}$ & $\frac{1}{3}$ &
\phantom{.}$(2,5,11)$ & $\begin{pmatrix} 5 & -2 & 0\\ 3 & 1 & -1 \end{pmatrix}$ & $\frac{4}{13}$ \\ \hline
\phantom{.}$(2,6,7)$ & $\begin{pmatrix} 3 & -1 & 0\\ 1 & 2 & -2 \end{pmatrix}$ & $\frac{4}{13}$ &
\phantom{.}$(2,6,9)$ & $\begin{pmatrix} 3 & -1 & 0\\ 0 & 3 & -2 \end{pmatrix}$ & $\frac{2}{7}$ \\ \hline
\phantom{.}$(2,7,8)$ & $\begin{pmatrix} 7 & -2 & 0\\ 4 & 0 & -1 \end{pmatrix}$ & $\frac{4}{15}$ &
\phantom{.}$(2,7,9)$ & $\begin{pmatrix} 7 & -2 & 0\\ 1 & 1 & -1 \end{pmatrix}$ & $\frac{3}{8}$ \\ \hline
\phantom{.}$(3,4,5)$ & $\begin{pmatrix} 4 & -3 & 0\\ -1 & 2 & -1 \end{pmatrix}$ & $\frac{1}{3}$ &
\phantom{.}$(3,4,6)$ & $\begin{pmatrix} 4 & -3 & 0\\ 2 & 0 & -1 \end{pmatrix}$ & $\frac{1}{3}$\\ \hline
\phantom{.}$(3,4,7)$ & $\begin{pmatrix} 4 & -3 & 0\\ 1 & 1 & -1 \end{pmatrix}$ & $\frac{2}{5}$ &
\phantom{.}$(3,4,8)$ & $\begin{pmatrix} 4 & -3 & 0\\ 0 & 2 & -1 \end{pmatrix}$ & $\frac{1}{3}$ \\ \hline
\phantom{.}$(3,4,9)$ & $\begin{pmatrix} 4 & -3 & 0\\ 3 & 0 & -1 \end{pmatrix}$ & $\frac{2}{7}$ &
\phantom{.}$(3,4,10)$ & $\begin{pmatrix} 4 & -3 & 0\\ 2 & 1 & -1 \end{pmatrix}$ & $\frac{4}{13}$ \\ \hline
\phantom{.}$(3,4,11)$ & $\begin{pmatrix} 4 & -3 & 0\\ 1 & 2 & -1 \end{pmatrix}$ & $\frac{1}{3}$ &
\phantom{.}$(3,5,6)$ & $\begin{pmatrix} 5 & -3 & 0\\ 2 & 0 & -1 \end{pmatrix}$ & $\frac{1}{3}$ \\ \hline
\phantom{.}$(3,5,7)$ & $\begin{pmatrix} 5 & -3 & 0\\ -1 & 2 & -1 \end{pmatrix}$ & $\frac{1}{3}$ &
\phantom{.}$(3,5,8)$ & $\begin{pmatrix} 5 & -3 & 0\\ 1 & 1 & -1 \end{pmatrix}$ & $\frac{5}{13}$ \\ \hline
\phantom{.}$(3,5,9)$ & $\begin{pmatrix} 5 & -3 & 0\\ 3 & 0 & -1 \end{pmatrix}$ & $\frac{8}{29}$ &
\phantom{.}$(3,5,10)$ & $\begin{pmatrix} 5 & -3 & 0\\ 0 & 2 & -1 \end{pmatrix}$ & $\frac{1}{3}$ \\ \hline
\phantom{.}$(3,6,7)$ & $\begin{pmatrix} 2 & -1 & 0\\ 1 & 3 & -3 \end{pmatrix}$ & $\frac{1}{3}$ &
\phantom{.}$(3,6,8)$ & $\begin{pmatrix} 2 & -1 & 0\\ 2 & 3 & -3 \end{pmatrix}$ & $\frac{1}{3}$ \\ \hline
\phantom{.}$(3,7,8)$ & $\begin{pmatrix} 1 & 3 & -3\\ 0 & 8 & -7 \end{pmatrix}$ & $\frac{3}{10}$ &
\phantom{.}$(4,5,6)$ & $\begin{pmatrix} 5 & -4 & 0\\ -1 & 2 & -1 \end{pmatrix}$ & $\frac{1}{3}$ \\ \hline
\phantom{.}$(4,5,7)$ & $\begin{pmatrix} 5 & -4 & 0\\ 3 & -1 & -1 \end{pmatrix}$ & $\frac{3}{11}$ &
\phantom{.}$(4,5,8)$ & $\begin{pmatrix} 5 & -4 & 0\\ 2 & 0 & -1 \end{pmatrix}$ & $\frac{1}{3}$ \\ \hline
\phantom{.}$(4,5,9)$ & $\begin{pmatrix} 5 & -4 & 0\\ 1 & 1 & -1 \end{pmatrix}$ & $\frac{5}{13}$ &
\phantom{.}$(4,6,7)$ & $\begin{pmatrix} 3 & -2 & 0\\ 2 & 1 & -2 \end{pmatrix}$ & $\frac{13}{47}$ \\ \hline
\phantom{.}$(5,6,7)$ & $\begin{pmatrix} 6 & -5 & 0\\ -1 & 2 & -1 \end{pmatrix}$ & $\frac{4}{13}$ & & & \\ \hline

\end{tabular}
\caption{The covering radius of the hexagons corresponding to the triples of velocities which sum to at most $18$ - Part III.}
\label{tbl:computer-examples-18-pt3}
\end{table}
  
\end{document}